\documentclass[11pt]{amsart}

\usepackage{amsmath, amssymb, amscd, cancel, graphicx, paralist, soul, stmaryrd}

\usepackage{mathdots}

\usepackage[all]{xy}
\usepackage{multirow}
\usepackage{longtable}
\usepackage{array}

\newtheorem{thm}{Theorem}

\newtheorem{lem}[thm]{Lemma}

\newtheorem{prop}[thm]{Proposition}

\def\Char{{\mathrm{Char}}}

\def\int{{\text{int}}}

\def\mod{{\textup{mod} \;}}

\begin{document}

\title[Changemaker lattices and Alexander polynomials of lens space knots]%
{A note on changemaker lattices and Alexander polynomials of lens space knots}

\author[Jacob Caudell]{Jacob Caudell}
\begin{abstract}
    We give an alternative proof of a recent theorem of Tange using the technology of changemaker lattices. Specifically, for $K\subset S^3$ a non-trivial knot with a lens space surgery, we give constraints on the Alexander polynomial of $K$ and the lens space surgery when the coefficient of $T^{g(K)-2}$ in $\Delta_K(T)$ is non-zero.
\end{abstract}

\address{Department of Mathematics, Boston College\\ Chestnut Hill, MA 02467}

\email{caudell@bc.edu}

\maketitle

\section{Introduction}
Let $K\subset S^3$ be a knot and $p/q\in \mathbb Q\cup\{1/0\}$ a \textit{slope} on the torus boundary of $S^3-\nu(K)$. Denote by $K(p/q)$ the result of $p/q$-framed Dehn surgery on $K$. A knot is said to be a \textit{lens space knot} if it admits a surgery to a lens space. Denote by $T_{r,s}$ the $(r,s)$-\textit{torus knot}.

It is well-known that the symmetrized Alexander polynomial of a lens space knot takes the form
\[
    \Delta_K(T) = (-1)^r +\sum_{j=1}^{r}(-1)^{j-1}(T^{n_j}+T^{-n_j}),
\]
for some sequence of integers $g=n_1>n_2>\ldots >n_r>0$ \cite[Corollary 1.3]{ozsvath2005knot}. Furthermore, it has been known for some time that $n_2=g-1$ \cite{hedden2018geography, tange2020alexander}. Much more recently, Tange proved the following:

\begin{thm}[\cite{tange2020term}]
Let $K\subset S^3$ be a non-trivial knot with $K(p)$ a lens space. If $n_3=g-2$, then $\Delta_K(T)=\Delta_{T_{2,2g+1}}(T)$, and $K(p)\cong T_{2,2g+1}(p)$.
\end{thm}
This theorem, an affirmative answer to a question of Teragaito, improves upon a list of criteria that Tange proved are each individually equivalent to the condition that $\Delta_K(T)=\Delta_{T_{2,2g+1}}(T)$ for $K$ a lens space knot \cite[Theorem 1.11]{tange2020alexander}. In this note, we offer an alternative proof of Theorem 1 using the technology of changemaker lattices.

\subsection{Conventions.} All manifolds are smooth and oriented. Let $U$ denote the unknot in $S^3$. The lens space $L(p,q)$ is oriented as $-U(p/q)$. Homology groups are taken with coefficients in $\mathbb Z$. For $L$ a lattice, $A_L$ denotes the Gram matrix for $L$ with respect to an implied basis.

\section*{Acknowledgments} Thanks to Fraser Binns and Braeden Reinoso for a helpful conversation. Thanks to Joshua Greene, the author's advisor, for his encouragement and for helpful feedback on this note.

\section{Preliminaries}
\subsection{Input from lattice theory.} Here we recall some elementary facts about lattices pertaining to lens space surgeries. 
\subsubsection{Linear lattices.} Let $p$ and $q$ be integers, with $p>q>0$. The improper fraction $p/q$ admits a unique continued fraction expansion

    $$p/q =[x_1,x_2,\ldots,x_n]^-= x_1 - \frac{1}{x_2-\frac{1}{\ddots- \frac{1}{x_n}}}$$
with each $x_i\geq 2$ an integer. The lens space $L(p,q)$ is the oriented boundary of the negative definite 4-manifold $X(p,q)$ given by attaching 4-dimensional 2-handles to a linear chain of $n$ unknots in the boundary of $B^4$, where the framing of the $i^\text{th}$ handle attachment is $-x_i$. The \textit{linear lattice} $\Lambda(p,q)$ is the free module $H_2(X(p,q))$ equipped with its intersection form. It is presented by the linking matrix of the surgery diagram in Figure \ref{fig:linearchain}, namely

\begin{figure}
    \centering
    \includegraphics[scale=.8]{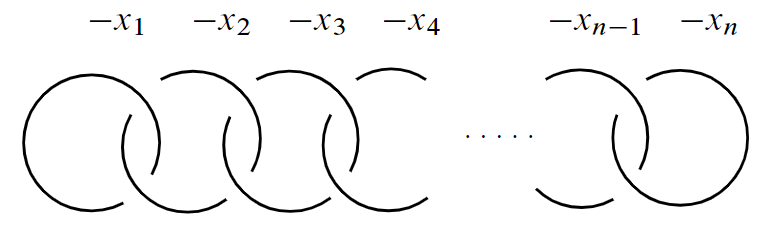}
    \caption{A Kirby diagram of $X(p,q)$.}
    \label{fig:linearchain}
\end{figure}

\[
    A_{\Lambda(p,q)} = \begin{bmatrix}
    -x_1 & 1 & 0 & \ & 0\\
    1 & -x_2 & 1 &   \ldots    & 0\\
    0 & 1 & -x_3 & \ & 0\\
    \ & \vdots & \ & \ddots & 1 \\
    0 & 0 & 0 & 1 & -x_n
    \end{bmatrix}.
\]

Denote by $x^{[k]}$ a list of $k$ $x's$, as in $[\ldots,2^{[3]},\ldots]^- = [\ldots, 2, 2, 2, \ldots]^-$.

\subsubsection{Changemaker lattices.} Let $\{e_0,e_1,\ldots, e_n\}$ be an orthonormal basis for $-\mathbb Z^{n+1}$. A vector $\sigma \in -\mathbb Z^{n+1}$ is said to be a \emph{changemaker} if $\sigma_0=0$ or $1$, and for every $1\leq i \leq n$,

\begin{equation}
    \sigma_{i-1}\leq \sigma_i\leq \sum_{j=0}^{i-1}\sigma_j +1.
\end{equation}
A negative definite rank $n$ lattice $L$ is said to be a \emph{changemaker lattice} if $L$ embeds in $-\mathbb Z^{n+1}$ as the orthogonal complement to a changemaker $\sigma$. 

There is an equivalent and coordinate-free characterization of the changemaker condition that is more suggestive of its context in Floer-theoretic considerations. Denote the \textit{characteristic elements} of $-\mathbb Z^{n+1}$ by $\Char(-\mathbb Z^{n+1})$, where $$\Char(-\mathbb Z^{n+1}) = \{\mathfrak c \in -\mathbb Z^{n+1} | \langle \mathfrak c, v\rangle \equiv \langle v, v \rangle \ (\mod 2)\} = \Big\{\sum_{i=0}^{n}\alpha_ie_i|\alpha_i \equiv 1 \ (\mod 2) \text{ for all } i\Big\}.$$
Denote by $\Char_i(-\mathbb Z^{n+1})$ the set of characteristic elements of self-pairing $-(n+1)-8i$, where elements of $\Char_0(-\mathbb Z^{n+1}) = \{\pm 1\}^{n+1}$ are commonly referred to as \textit{short} characteristic elements. Then, $\sigma$ is a changemaker if and only if $\{\langle \mathfrak c, \sigma\rangle | \mathfrak c \in \Char_0(-\mathbb Z^{n+1})\}$ contains all integers $j$ with $j \equiv \langle \sigma, \sigma\rangle \ (\mod 2)$ and $|j| \leq |\sigma|_1$ \cite[Proposition 3.1]{greene2015space}.

\subsection{Input from Heegaard Floer homology.} The utility of changemaker lattices relies on the relationship between the Alexander polynomial of a lens space knot $K\subset S^3$, $HF^+(K(0))$, and the \emph{correction terms} of $K(p)$ \cite[Theorem 6.1]{MR2956253}. For $i$ a non-negative integer, define the $i^{th}$ \emph{torsion coefficient} of a knot $K\subset S^3$ with Alexander polynomial $\Delta_K(T)=\sum a_jT^j$ by  $$t_i(K) = \sum_{j=1}^\infty ja_{i+j}.$$
The following proposition contains the facts about torsion coefficients we will need in our proof.

\begin{prop}
Let $K\subset S^3$ be a lens space knot (or more generally an L-space knot).
\begin{itemize}
    \item $\{t_i(K)\}$ is a non-increasing sequence of non-negative integers that determines $\Delta_K(T)$.
    \item $t_i(K)=0$ if and only if $g(K)\leq i$.
    \item $t_i(K)=1$ if and only if $n_3\leq i \leq g(K)-1$.
\end{itemize}
\end{prop}
\begin{proof}
Recall that
\[
    \Delta_K(T) = (-1)^r +\sum_{j=1}^{r}(-1)^{j-1}(T^{n_j}+T^{-n_j}), \ n_1 = g(K), \ n_2 = g(K)-1.
\]
All of these facts then follow from the observation that $$t_i(K)-t_{i+1}(K) = \sum_{\{j|n_j\geq i +1\}}(-1)^{j-1},$$ which is either $0$ or $1$.
\end{proof}

We now state, without proof, specializations of \cite[Lemma 2.5 and Theorem 3.3]{greene2015space} in the case $K(p) \cong L(p,q)$.

\begin{thm}
If $K(p)\cong L(p,q)$, then $\Lambda(p,q)$ embeds as the orthogonal complement to a changemaker $\sigma\in -\mathbb{Z}^{n+1}$ with $\langle \sigma, \sigma\rangle =-p$ and $\sigma_0=1$. Moreover, 
\begin{equation}
    \mathfrak c^2 + (n+1)\leq -8t_i(K)
\end{equation}
for all $|i|\leq p/2$ and $\mathfrak c\in \Char(-\mathbb Z^{n+1})$ such that $\langle \mathfrak c, \sigma\rangle + p \equiv 2i \ (\mod 2p).$ Furthermore, for each $|i|\leq p/2$ there exists $\mathfrak c$ attaining equality in (2).
\end{thm}

\section{An alternative proof of Theorem 1} 
Our proof of Theorem 1 follows from Lemmas 4 and 5. 

\begin{lem}
If $K$ is non-trivial, $K(p)\cong L(p,q)$, and $n_3=g(K)-2$, then $\sigma_n=2$.
\end{lem}

\begin{lem}
If $\sigma$ is a changemaker with $\sigma_n = 2$ and $(\sigma) ^\perp$ is a linear lattice, then $\sigma = (1, 2, \ldots, 2)$ or $\sigma=(1,1,1,2, \ldots, 2)$.
\end{lem}

\begin{proof}[Proof of Theorem 1]

Observe that

\[A_{(1,2,\ldots,2)^\perp}=\begin{bmatrix}
-5 & 1 & 0 & \ & 0 & 0\\
1 & -2 & 1& \ldots & 0&0\\
0 & 1 & -2 & \ & 0&0\\
\ & \vdots & \ & \ddots & \vdots& \ \\
0 & 0 & 0 & \ldots &-2 & 1\\
0 & 0 & 0 & \ & 1 & -2
\end{bmatrix}, \ 
A_{(1,1,1,2,\ldots,2)^\perp}=\begin{bmatrix}
-2 & 1 & 0 & 0 & \ & 0\\
1 & -2 & 1& 0 & \ldots &0\\
0 & 1 & -3 & 1 & \ &0\\
 0& 0 & 1 & -2 & \ &\vdots \\
\ & \vdots & \ & \ &\ddots & 1\\
0 & 0 & 0 & \ldots & 1 & -2
\end{bmatrix}.
\]
Let $g=g(K)$. The torus knot $T_{2, 2g+1}$ admits two integral Dehn surgeries to lens spaces: $T_{2,2g+1}(4g+1) \cong L(4g+1, g)$ and $T_{2,2g+1}(4g+3)\cong L(4g+3,3g+2)$. Furthermore, $(4g+1)/g = [5, 2^{[g-1]}]^-$ and $(4g+3)/(3g+2)=[2,2,3,2^{[g-1]}]^-$. By \cite[Theorem 3]{gerstein1995nearly}, if $\Lambda(p,q)\cong\Lambda(p',q')$ then $p =p'$ and $q=q'$ or $qq'\equiv 1 \ (\mod p)$ (cf. \cite[Proposition 3.6]{greene2015space}). Since $\sigma$ then determines both $\{t_i(K)\}$ and the oriented homeomorphism type of $K(p)$, we conclude that $\Delta_K(T)=\Delta_{T_{2,2g+1}}(T)$, $p= 4g+1$ or $p=4g+3$, and $K(p)\cong T_{2,2g+1}(p)$.
\end{proof}

\begin{proof}[Proof of Lemma 4] For $0 \leq i \leq p/2$, Theorem 3 allows us to compute $$t_i(K)=\min\{k|\langle \mathfrak c , \sigma \rangle + p \equiv 2i\ (\mod 2p) \text{ for some } \mathfrak c \in \Char_k(-\mathbb Z^{n+1})\}.$$ 

Suppose that $\sigma_n\geq 3$. Let $t\in \{ 0, 1, \ldots, n\}$ be the minimum index for which $\sigma_i\geq 3$. Since $\sigma$ is a changemaker, there exists $A\subset \{0, 1, \ldots, n\}$ such that $\sum_{k \in A}\sigma_k = \sigma_t-3$. We point out that $t\not \in A$. Let $\mathfrak c = \sum_{j\not\in A}e_j - \sum_{k\in A}e_k + 2e_t$. Then

$$p + \langle \mathfrak c,\sigma\rangle = p -|\sigma|_1+2\sum_{k\in A}\sigma_k -2 \sigma_t = p -|\sigma|_1 - 6 = 2g-6,$$
so $t_{g-3}(K)= 1$. Since $n_3 = g-2$, $t_i(K) = 1$ if and only if $g-2\leq i \leq g-1$, and therefore $\sigma_n\leq 2$. Now, if $\sigma_n = 1$, then $p=|\langle \sigma, \sigma\rangle| = n+1$ and $g=\frac{p -|\sigma|_1}{2} = 0$, so $K=U$. We conclude that $\sigma_n = 2$. 
\end{proof}

\begin{proof}[Proof of Lemma 5]
Since $\sigma_n= 2$, there exists some index $1\leq k\leq n$ for which $\sigma_i=1$ if $i\leq k$ and $\sigma_j =2$ if $k < j$. If $k\geq 2$, then $(\sigma)^\perp$ admits the \emph{standard basis} $S= (v_1, \ldots, v_n)$, where $v_i = e_{i-1}-e_i$ for $i\neq k$, and $v_k = e_{k-2} + e_{k-1} -e_k$. Consider the \emph{intersection graph} on $S$, $G(S)=(V, E)$ where $V= S$, and $(v_i,v_j)\in E$ if $|\langle v_i, v_j\rangle|= 1$. (This is not the usual definition of the edge set of an intersection graph, but is equivalent given that $|\langle v_i,v_i\rangle| =2$ for all but one element of $S$.) If $k \geq 4$, then the subgraph induced by $\{v_{k-3}, v_{k-2},v_{k-1},v_k\}$ is a \emph{claw}, which cannot occur as an induced subgraph of $G(S)$ if $(\sigma)^\perp$ is a linear lattice \cite[Lemma 4.8]{greene2013lens}. If $k = 2$, then $(\sigma)^\perp$ is decomposable. It is straightforward to verify that $(\sigma)^\perp$ is a linear lattice in either of the cases $k = 1$ or $k=3$, as we have done in the proof of Theorem 1.
\end{proof}

\bibliographystyle{plain}
\bibliography{ChangemakersAndAlexanderPolynomials}

\end{document}